\documentclass[a4paper,12pt]{article}
\usepackage[T1]{fontenc}
\usepackage{mathrsfs}
\usepackage{graphicx}
\usepackage{enumerate}
\usepackage{multicol}
\usepackage{color}
\usepackage{amsmath,amssymb}
\usepackage{amsthm}
\usepackage{textcomp}
\usepackage{mathabx}

\usepackage{hyperref}

\usepackage{mathpazo}

\usepackage[top=1in, bottom=1in, left=0.75in, right=0.75in]{geometry}

\theoremstyle{plain}
\newtheorem{thm}{Theorem}
\newtheorem{lem}{Lemma}
\newtheorem{cor}{Corollary}
\newtheorem{prop}{Proposition}

\theoremstyle{definition}
\newtheorem{dfn}{Definition}

\theoremstyle{remark}
\newtheorem*{rem}{Remark}
\newtheorem*{opp}{Open Problem}
\newtheorem*{ackn}{Acknowledgment}

\title{Perturbations of Cauchy differences}
\author{Eszter Gselmann, Tomasz Małolepszy and Janusz Matkowski}

\begin{document}
\nocite{*}

\maketitle

\begin{abstract}
 This paper investigates functional equations arising from perturbations of Cauchy differences. We study equations of the form
 \[
f(x+y)-f(x)-f(y)=B(x,y) \quad \text{or} \quad f(xy)-f(x)f(y) = B(x,y)
\]
where $B$ is a biadditive mapping, and also more general cases where the inhomogeneity depends on unknown functions
\begin{align*}
 f(x+y)-f(x)-f(y)&= \alpha x y \\[2.5mm]
 f(x+y)-f(x)-f(y)&= \alpha (x y)\\[2.5mm]
 f(x+y)-f(x)-f(y)&= \alpha(x)\alpha(y).   
\end{align*}
Our results extend previous work on the bilinearity of the Cauchy exponential difference by Alzer and Matkowski. We characterize solutions under various structural and regularity assumptions, including additive and exponential Cauchy differences, and show that solutions often reduce to additive functions, exponential polynomials, or combinations thereof. For Levi-Civita type equations, we provide explicit representations of solutions in terms of additive and exponential components. Furthermore, we determine conditions under which real-valued solutions exist and describe their forms. The paper concludes with open problems concerning generalized equations that cannot be solved by the methods presented here, suggesting directions for future research.
\end{abstract}

\section{Introduction}

The recent paper \cite{AlzMat24} investigates a functional equation that represents a bilinear perturbation of Cauchy’s classical exponential equation. 
The authors consider a modified form where the Cauchy difference equals a bilinear term
\[
f(x)f(y) - f(x+y) = \alpha xy 
\]
for the function $f\colon \mathbb{R}\to \mathbb{R}$, where $\alpha$ is a given nonzero real number. 

The paper characterizes all solutions of this equation under mild regularity assumptions. If \(f \colon \mathbb{R} \to \mathbb{R}\) satisfies the above equation and is differentiable at least at one point, then the only solutions are functions of the form
\[
f(x) = 1 \pm \sqrt{\alpha}\, x,
\]
where \(\alpha > 0\). The authors also formulate a conjecture stating that every solution of the equation must have a zero. 
Further, \cite{AlzMat24} examines related functional equations such as
\[
f(x) + f(y) - f(x+y) = \alpha xy,
\]
and shows that these can be reduced to additive or Pexider-type equations. Continuous solutions are described using additive functions. The methods employed include classical techniques from the theory of functional equations, differentiation, and reduction to simpler forms.

The paper \cite{FecPieSme25} addresses a conjecture recently posed by Horst Alzer and Janusz Matkowski in \cite{AlzMat24} concerning a bilinearity property of the Cauchy exponential difference. In \cite{AlzMat24} it was formulated as an open question whether every solution \(f \colon \mathbb{R} \to \mathbb{R}\) of the above equation has a zero.

The authors of the paper \cite{FecPieSme25} provide an affirmative answer to this conjecture: every solution indeed has a zero. Furthermore, they extend the problem to a more general setting by considering mappings on a linear space \(X\) over \(\mathbb{K} \in \{\mathbb{R}, \mathbb{C}\}\) and a biadditive functional \(\varphi \colon  X \times X \to \mathbb{K}\). The generalized equation takes the form
\[
f(x+y) = f(x)f(y) - \varphi(x,y), \quad x,y \in X.
\]
The paper characterizes solutions under various assumptions on \(\varphi\), proving that the structure of solutions depends on the behavior of the biadditive functional. In particular, they show that if \(\varphi(z_0,z_0) \neq 0\) for some \(z_0 \in X\), then solutions are of the form
\[
f(x) = a \varphi(x,z_0) + 1,
\]
where \(a\) satisfies certain quadratic conditions involving \(\varphi\). The authors also provide corollaries describing when no solutions exist, depending on sign conditions on \(\varphi\).

This work not only resolves the original conjecture but also opens new directions by generalizing the problem to linear spaces and biadditive perturbations. It contributes to the theory of functional equations, particularly those involving bilinear modifications of exponential-type equations.

Several classical papers provide the foundation for the study of Cauchy differences and their generalizations. \cite{BarVol91} revisited van der Corput’s theorem in the context of functional equations, focusing on additive and modular structures. Their work established conditions for solutions of Cauchy-type equations and inspired later research on bilinear perturbations.
The paper \cite{BarKan93} studied the Cauchy differences
and characterized solutions under various constraints. 
In \cite{BarKom03} Karol Baron and Zygfryd Kominek studied functionals where the Cauchy difference is bounded by a homogeneous functional. Their results highlight structural restrictions on solutions, which parallels our approach of imposing biadditivity conditions on the perturbation term.
Finally, in the paper \cite{EbaKanSah92}, the authors examined Cauchy differences depending on the product of arguments. This idea anticipates the multiplicative and Levi-Civita type equations treated in our manuscript, where the inhomogeneity involves products or exponential components.

The recent manuscipt \cite{MalMat25} investigates when real Cauchy differences can be expressed as biadditive functions, that is, as products of two single-variable additive functions. The paper considers four types of Cauchy differences: exponential, additive, logarithmic, and multiplicative, and examines conditions under which such representations exist.
For the exponential Cauchy difference, the authors improve earlier results by removing differentiability assumptions. In the case of additive Cauchy differences, they provide explicit formulas for solutions, which involve additive functions and quadratic terms in $\alpha(x)$. For logarithmic Cauchy differences, the paper proves that no solutions exist either on the real line or on the positive half-line.

In this paper, we aim to contribute to the abovementioned topics under more general conditions compared to previous works. In the results presented in the second chapter, the domain of definition is a group, in some cases a semigroup, while the range is a field, in certain cases the field of complex numbers or the field of real numbers.

At first we consider functions $f\colon S\to \mathbb{F}$ that fulfill 
\[
 f(x+y)-f(x)-f(y)= B(x, y)
\]
for all $x, y\in S$, where $B\colon S\times S\to \mathbb{F}$ is a given biadditive function, and also the related equations
\begin{align*}
 f(x+y)-f(x)-f(y)&= \alpha x y \\[2.5mm]
 f(x+y)-f(x)-f(y)&= \alpha (x y)\\[2.5mm]
 f(x+y)-f(x)-f(y)&= \alpha(x)(y),  
\end{align*}
respectively. 
In the case of the latter two equations, both the functions $f$ and $\alpha$ are unknown functions. 
Further, not only 'additive' Cauchy differences, but also 'exponential' Cauchy differences as well as  the related equations are studied. 

\subsection*{Prerequisites related to the Levi-Civita equation}

In this subsection, we recall the most important concepts and results related to exponential polynomials, with particular attention to the Levi-Civita equation and its solutions. In presenting these concepts and results, we will follow László Székelyhidi’s monograph \cite{Sze91}.

In what follows, $(G, \cdot)$ is assumed to be  a commutative group (written multiplicatively).
Further, $\mathbb{N}$ wil denote the set of positive integers, $\mathbb{C}$ will stand for the set of complex numbers, and $\mathbb{R}$ denotes the set of real numbers. 

\begin{dfn}
 The function $a\colon G\to \mathbb{C}$ is called \emph{additive} if 
 \[
  a(x\cdot y)= a(x)+a(y)
 \]
holds for all $x, y\in G$. 
\end{dfn}

\begin{dfn}
{\it Polynomials} are elements of the algebra generated by additive
functions over $G$. More exactely, if $n$ is a positive integer,
$P\colon\mathbb{C}^{n}\to \mathbb{C}$ is a (classical) complex
polynomial in
 $n$ variables and $a_{k}\colon G\to \mathbb{C}\; (k=1, \ldots, n)$ are additive functions, then the function
 \[
  x\longmapsto P(a_{1}(x), \ldots, a_{n}(x))
 \]
is a polynomial and, also conversely, every polynomial can be
represented in such a form.
\end{dfn}

\begin{rem}
 We recall that the elements of $\mathbb{N}^{n}$ for any positive integer $n$ are called
 ($n$-dimensional) \emph{multi-indices}.
 Addition, multiplication and inequalities between multi-indices of the same dimension are defined component-wise.
 Further, we define $x^{\alpha}$ for any $n$-dimensional multi-index $\alpha$ and for any
 $x=(x_{1}, \ldots, x_{n})$ in $\mathbb{C}^{n}$ by
 \[
  x^{\alpha}=\prod_{i=1}^{n}x_{i}^{\alpha_{i}}
 \]
where we always adopt the convention $0^{0}=0$. We also use the
notation $\left|\alpha\right|= \alpha_{1}+\cdots+\alpha_{n}$. With
these notations any polynomial of degree at most $N$ on the
commutative semigroup $G$ has the form
\[
 p(x)= \sum_{\left|\alpha\right|\leq N}c_{\alpha}a(x)^{\alpha}
 \qquad
 \left(x\in G\right),
\]
where $c_{\alpha}\in \mathbb{C}$ and $a=(a_1, \dots, a_n) \colon
G\to \mathbb{C}^{n}$ is an additive function. Furthermore, the
\emph{homogeneous term of degree $k$} of $p$ is
\[
 \sum_{\left|\alpha\right|=k}c_{\alpha}a(x)^{\alpha} .
\]
\end{rem}

\begin{lem}[Lemma 2.7 of \cite{Sze91}]\label{L_lin_dep}
 Let $G$ be a commutative group,
 $n$ be a positive integer and let
 \[
  a=\left(a_{1}, \ldots, a_{n}\right),
 \]
where $a_{1}, \ldots, a_{n}$ are linearly independent complex-valued
additive functions defined on $G$. Then the monomials
$\left\{a^{\alpha}\right\}$ for different multi-indices are linearly
independent.
\end{lem}

\begin{dfn}
A function $m\colon G\to \mathbb{C}$ is called an \emph{exponential}
function if it satisfies
\[
 m(x\cdot y)=m(x)m(y)
 \qquad
 \left(x,y\in G\right).
\]
Furthermore, on an  \emph{exponential polynomial} we mean a linear
combination of functions of the form $p \cdot m$, where $p$ is a
polynomial and $m$ is an exponential function.
\end{dfn}

\begin{dfn}
 Let $G$ be an Abelian group and $V\subseteq \mathbb{C}^G$ a set of complex-valued functions defined on $G$. We say that $V$ is \emph{translation-invariant} if for every $f\in V$ and for all $g\in G$, we also have $\tau_{g}f\in V$,   where
 \[
  \tau_{g}f(h)= f(hg)
  \qquad
  \left(h\in G\right).
 \]
 \end{dfn}

 In view of Theorem 10.1 of Székelyhidi \cite{Sze91}, any finite dimensional translation-invariant linear 
 space of complex-valued functions on a commutative group consists of exponential polynomials. 
 This implies that if $G$ is a commutative group, then any function 
 $f\colon G\to \mathbb{C}$, satisfying the functional equation 
 \[
  f(x\cdot y)= \sum_{i=1}^{n}g_{i}(x)h_{i}(y) 
  \qquad 
  \left(x, y\in G\right)
 \]
for some positive integer $n$ and functions $g_{i}, h_{i}\colon G\to \mathbb{C}$ ($i=1, \ldots, n$), 
is an exponential polynomial of degree at most $n$. 

If the functions $h_{1}, \ldots, h_{n}$ are linearly independent, then $g_{1}, \ldots, g_{n}$ are linear combinations of translates of $f$, hence they are exponential polynomials of order at most $n$, too. Moreover, they are built up from the same additive and exponential functions, as $f$. If either the functions $g_{1}, \ldots, g_{n}$, or the functions $h_{1}, \ldots, h_{n}$ are linearly dependent, then the number $n$ in the above representation can be reduced. We will call this case \emph{degenerate}. 

Summing up, in the \emph{nondegenerate} case, all the functions $f$, $g_{1}, \ldots, g_{n}$ and $h_{1}, \ldots, h_{n}$ are exponential polynomials, that is, we have 
\begin{align*}
 f(x)&= \sum_{j=1}^{k}P_{j}(a_{j, 1}(x), \ldots, a_{j, n_{j}-1}(x)) m_{j}(x)& \qquad 
 \left(x\in G\right)\\
 g_{i}(x)&= \sum_{j=1}^{k}Q_{i, j}(a_{j, 1}(x), \ldots, a_{j, n_{j}-1}(x)) m_{j}(x)&\qquad 
 \left(x\in G, i=1, \ldots, n\right)\\
 h_{i}(x)&= \sum_{j=1}^{k}R_{i, j}(a_{j, 1}(x), \ldots, a_{j, n_{j}-1}(x)) m_{j}(x)& 
 \qquad 
 \left(x\in G, i=1, \ldots, n\right), 
\end{align*}
where $k, n_{1}, \ldots, n_{k}$ are positive integers, $m_{1}, \ldots, m_{k}$ are different complex-valued exponentials, $\{a_{j, 1}, \ldots, a_{j, n_{j}-1} \}$ are sets of linearly independent additive functions for $j=1, \ldots, k$, $P_{j}, Q_{i, j}, R_{i, j}\colon \mathbb{C}^{n_{j}-1}\to \mathbb{C}$ are complex polynomials for $i=1, \ldots, n$ and $j=1, \ldots, k$ such that we have $n_{1}+\cdots+n_{k}=n$. Of course, the polynomials  $P_{j}, Q_{i, j}, R_{i, j}$ must satisfy a lot of extra conditions in order to form a solution for the equation in question.

\section{Results}

\subsection{Additive Cauchy differences}

\begin{prop}\label{prop_biadditive}
 Let $(S, +)$ be a semigroup, $(\mathbb{F},+, \cdot, )$ be a field, and $B\colon S\times S\to \mathbb{F}$ be a symmetric and biadditive function. If a function $f\colon S\to \mathbb{F}$ satisfies equation 
 \begin{equation}\label{biadditive}
  f(x+y)-f(x)-f(y)= B(x, y)
 \end{equation}
for all $x, y\in S$, then exists an additive function $a\colon S\to \mathbb{F}$ i.e., a function for which 
\[
 a(x+y)= a(x)+a(y)
\]
holds for all $x, y\in S$ such that 
\[
 f(x)= \frac{1}{2}B(x, x)+a(x)
\]
for all $x\in S$. 
\end{prop}

\begin{proof}
 Since the mapping $B\colon S\times S\to \mathbb{F}$ is symmetric and biadditive, we have 
 \[
  B(x, y)= \frac{1}{2}\left[B(x+y, x+y)-B(x, x)-B(y, y)\right]
 \]
for all $x, y\in S$. Using this, equation \eqref{biadditive} can be rewritten as 
\[
 f(x+y)-f(x)-f(y)= \frac{1}{2}\left[B(x+y, x+y)-B(x, x)-B(y, y)\right] 
 \qquad 
 \left(x, y\in S\right). 
\]
Define the function $a\colon S\to \mathbb{F}$ as 
\[
 a(x)= f(x)-\frac{1}{2}B(x, x) 
 \qquad 
 \left(x\in S\right), 
\]
to deduce that 
\[
 a(x+y)=a(x)+a(y)
\]
holds for all $x, y\in S$. So the function $a$ is additive. Thus the function $f$ can be written as 
\[
 f(x)= \frac{1}{2}B(x, x)+a(x) 
 \qquad 
 \left(x\in S\right)
\]
with the aid of an additive function $a\colon S\to \mathbb{F}$. 
\end{proof}

\begin{cor}\label{cor1}
 Let $(\mathbb{F},+, \cdot)$ be a field such that 
 $(S, +)$ is a subsemigroup of the group $(\mathbb{F}, +)$ and let $\alpha \in \mathbb{F}$ be arbitrarily fixed. If a function $f\colon S\to \mathbb{F}$ satisfies equation 
 \begin{equation}\label{additive_difference}
  f(x+y)-f(x)-f(y)= \alpha xy
 \end{equation}
for all $x, y\in S$, then there exists an additive function $a\colon S\to \mathbb{F}$ such that 
\[
 f(x)= \frac{\alpha}{2}x^{2}+a(x)
\]
holds for all $x\in S$. 
\end{cor}

\begin{cor}\label{lem_exp_difference}
 Let $(\mathbb{F},+, \cdot, )$ be a field such that 
 $(S, \cdot)$ is a subsemigroup of the semigroup $(\mathbb{F}, \cdot)$ and let $\alpha \in \mathbb{F}$ be arbitrarily fixed. If a function $f\colon S\to \mathbb{F}$ satisfies equation 
 \begin{equation}\label{exp_difference}
  f(xy)-f(x)-f(y)= \alpha xy
 \end{equation}
for all $x, y\in S$, then necessarily $\alpha=0$ and the function $f\colon S\to \mathbb{F}$ is additive, i.e., 
\[
 f(x y)= f(x)+f(y)
\]
holds for all $x, y\in S$. 
\end{cor}

\begin{proof}
 Under the assumptions of the proposition, let us consider the function $\mathscr{C}_{f}\colon S\times S\to \mathbb{F}$ defined by 
 \[
\mathscr{C}_{f}(x, y)= f(xy)-f(x)-f(y)  
\qquad 
\left(x, y\in S\right). 
 \]
 As $\mathscr{C}_{f}$ is a Cauchy difference generated by the function $f$, it fulfills the cocycle equation on $S\times S$, i.e., we have 
 \[
  \mathscr{C}_{f}(x, y)+\mathscr{C}_{f}(xy, z)= \mathscr{C}_{f}(x, yz)+\mathscr{C}_{f}(y, z) 
  \qquad 
  \left(x, y, z\in S\right). 
 \]
Due to equation \eqref{exp_difference}, we obtain that the function $\mathscr{B}\colon S\times S\to \mathbb{F}$, defined by 
\[
 \mathscr{B}(x, y)= \alpha xy 
 \qquad 
 \left(x, y\in S\right)
\]
also satisfies the cocycle equation, since we have $\mathscr{C}_{f}\equiv \mathscr{B}$. This however means that 
\[
 \alpha xy+\alpha xyz= \alpha xyz+\alpha yz
\]
should hold for all $x, y, z\in S$. From which we get that 
\[
 \alpha y (x-z)=0 
\]
for all $x, y, z\in S$. As $S \subset \mathbb{F}$, we obtain that the above identity can hold for all $x, y, z \in S$ only if $\alpha=0$. 
\end{proof}

\begin{rem}
 Note that in Proposition \ref{prop_biadditive} and also in Lemma \ref{lem_exp_difference}, it is sufficient for 
$S$ to be merely a semigroup; $S$ does not necessarily have to be a group. For example, if we choose $S=]0, +\infty[$ and the semigroup operation is either addition or multiplication, we can similarly obtain the solutions to the above equations. Moreover, it immediately follows that every solution can be uniquely extended to the entire real line.
\end{rem}

Taking $(S, +)= (\mathbb{R}, +)$ and $(\mathbb{F},+, \cdot)= (\mathbb{R},+, \cdot )$, the corollary below follows immediately. 

\begin{cor}\label{cor_add}
 Let $\alpha\in \mathbb{R}$ be fixed. If a function $f\colon \mathbb{R}\to \mathbb{R}$ satisfies equation 
 \begin{equation}\label{eq_add}
  f(x+y)-f(x)-f(y)= \alpha xy
  \qquad 
  (x, y \in \mathbb{R}), 
 \end{equation}
 then there exists an additive function $a\colon \mathbb{R}\to \mathbb{R}$ i.e., a function $a$ for which 
\[
 a(x+y)= a(x)+a(y)
\]
holds for all $x, y\in \mathbb{R}$ such that 
\[
 f(x)= \frac{\alpha}{2}x^{2}+a(x) 
\]
for all $x\in \mathbb{R}$. 
\end{cor}

Let $\mathbb{R}^{\times}$ denote the set of nonzero real numbers, i.e., 
$\mathbb{R}^{\times}= \{ x\in \mathbb{R}\, \vert \, x\neq 0\}$.

Similarly, if $D\in \{\mathbb{R}^{\times}, ]0, +\infty[ \}$  and  $(S, \cdot)= (D, \cdot)$ and $(\mathbb{F},+, \cdot)= (\mathbb{R},+, \cdot )$, the corollary below follows immediately. 

\begin{cor}\label{cor_exp_positive}
 Let $D\in \{\mathbb{R}^{\times}, ]0, +\infty[ \}$ and  $\alpha\in \mathbb{R}$ be fixed. If a function $f\colon D\to \mathbb{R}$ satisfies equation 
 \begin{equation}\label{eq_exp}
  f(xy)-f(x)-f(y)= \alpha xy
  \qquad 
  (x, y \in D), 
 \end{equation}
 then $\alpha=0$ and $f$ is a logarithmic function on $D$, i.e., 
\[
 f(xy)= f(x)+f(y)
\]
holds for all $x, y\in D$. 
\end{cor}

Note that we can also consider functions $f\colon \mathbb{R}\to \mathbb{R}$ and assume that equation \eqref{eq_exp} to hold for all $x, y\in \mathbb{R}$. 
In this case, however, the set of solutions is significantly narrower. More precisely, only the identically zero function will be a solution to this equation. The reason is that if a function $f\colon \mathbb{R}\to \mathbb{R}$ satisfies  equation
\[
 f(xy)= f(x)+f(y)
\]
for all $x, y\in \mathbb{R}$, then $f$ must necessarily be identically zero. Therefore, whether $0$ is included (or not) in the domain of the involved function, plays an important role.

\begin{cor}\label{cor_exp_real}
 Let $\alpha\in \mathbb{R}$ be fixed. If a function $f\colon \mathbb{R}\to \mathbb{R}$ satisfies equation 
 \begin{equation}\label{eq_exp_real}
  f(xy)-f(x)-f(y)= \alpha xy
  \qquad 
  (x, y \in \mathbb{R}), 
 \end{equation}
 then $\alpha=0$ and $f$ is identically zero on $\mathbb{R}$. 
\end{cor}

As a generalization of equation \eqref{additive_difference}, we can obtain the following statement with a completely similar proof. However, here it is important that $S$ is not just a semigroup, but a group and $\mathbb{F}^{\times}$ denotes the nonzero elements of  $\mathbb{F}$, i.e., 
$\mathbb{F}^{\times}= \{ x\in \mathbb{F}\, \vert \, x\neq 0\}$. 

\begin{prop}\label{prop2}
  Let  $(\mathbb{F},+, \cdot, )$ be a field and $(S, \cdot)$ be a subgroup of $(\mathbb{F}^{\times}, \cdot)$.  If the functions $f, \alpha\colon S\to \mathbb{F}$ satisfy equation 
 \begin{equation}\label{additive_difference_func}
  f(x\cdot y)-f(x)-f(y)= \alpha(xy)
 \end{equation}
for all $x, y\in S$, then there exists an additive function $a\colon S\to \mathbb{F}$, i.e., a function for which we have     
\[
 a(x \cdot y)= a(x)+a(y)
\]
for all $x, y\in S$ such that 
\[
 \alpha(x)= -f(1) 
 \qquad 
 \text{and}
 \qquad 
 f(x)= a(x)+f(1) 
 \qquad 
 \left(x\in S\right). 
\]
\end{prop}

\begin{proof}
 Assume that the functions $f, \alpha, \colon S\to \mathbb{F}$ satisfy equation \eqref{additive_difference_func}. 
 Setting $y=1$ in \eqref{additive_difference_func}, 
 \[
 \alpha(x)= -f(1)
 \]
 follows for all $x\in S$. So $\alpha$ is a constant function. 
 In view of this, equation \eqref{additive_difference_func} reads as 
 \[
 f(x \cdot y)-f(x)-f(y)= -f(1)
 \qquad 
 \left(x, y\in S\right). 
 \]
 This shows that for the function $a\colon S\to \mathbb{F}$ defined by 
 \[
 a(x)= f(x)-f(1)
 \qquad 
 \left(x\in S\right), 
 \]
 we have 
 \[
 a(x\cdot y)= a(x)+a(y)
 \qquad 
 \left(x, y\in S\right). 
 \]
 So the function $a$ is additive on $S$, and we have 
 \[
 f(x)= a(x)+f(1)
 \qquad 
 \left(x\in S\right). 
 \]
\end{proof}

Below we study the 'additive' counterpart of the previously studied equation. 

\begin{prop}\label{prop3}
 Let $(\mathbb{F}, +, \cdot)$ be a field. If the functions $f, \alpha \colon \mathbb{F}\to \mathbb{F}$ satisfy 
 \begin{equation}\label{add_product}
  f(x+y)-f(x)-f(y)= \alpha(xy)
 \end{equation}
for all $x, y\in \mathbb{F}$, then there exist a constant $\gamma\in \mathbb{F}$, and there are additive functions $a, A\colon \mathbb{F}\to \mathbb{F}$, i.e., functions for which 
\[
 a(x+y)= a(x)+a(y) 
 \quad 
 \text{and}
 \quad 
 A(x+y)= A(x)+A(y) 
 \qquad 
 \left(x, y\in \mathbb{F}\right)
\]
such that 
\[
 \alpha(x)= a(x)+\gamma 
 \quad 
 \text{and}
 \quad 
 f(x)= \frac{1}{2}a(x^{2})+A(x)+\gamma
\]
holds for all $x\in \mathbb{F}$. 
\end{prop}

\begin{proof}
 Let us consider the function $\mathscr{C}_{f}\colon \mathbb{F}\times \mathbb{F}\to \mathbb{F}$ defined by 
 \[
\mathscr{C}_{f}(x, y)= f(x+y)-f(x)-f(y)  
\qquad 
\left(x, y\in \mathbb{F}\right). 
 \]
 As $\mathscr{C}_{f}$ is a Cauchy difference generated by the function $f$, it fulfills the cocycle equation on $\mathbb{F}\times \mathbb{F}$, i.e., we have 
 \[
  \mathscr{C}_{f}(x, y)+\mathscr{C}_{f}(x+y, z)= \mathscr{C}_{f}(x, y+z)+\mathscr{C}_{f}(y, z) 
  \qquad 
  \left(x, y, z\in \mathbb{F}\right). 
 \]
 Further, due to our equation, we have $\mathscr{C}_{f}(x, y)= \alpha(xy)$ for all $x, y\in \mathbb{F}$. Therefore 
 \[
  \alpha(xy)+\alpha((x+y)z)= \alpha(x(y+z))+\alpha(yz) 
  \qquad 
  \left(x, y, z\in \mathbb{F}\right), 
 \]
or after some rearrangement 
\[
 \alpha(xy)+\alpha(xz+yz)= \alpha(xy+xz)+\alpha(yz). 
\]
In other words, we have 
\[
 \alpha(u)+\alpha(v+w)= \alpha(u+v)+\alpha(w) 
 \qquad 
 \left(u, v, w\in \mathbb{F}\right). 
\]
Therefore, there exists a constant $\gamma\in \mathbb{F}$ and an additive function $a\colon \mathbb{F}\to \mathbb{F}$ such that 
\[
 \alpha(x)= a(x)+\gamma 
 \qquad 
 \left(x\in \mathbb{F}\right). 
\]
Define the function $B\colon \mathbb{F}\times \mathbb{F}\to \mathbb{F}$ through 
\[
 B(x, y)= a (xy) 
 \qquad 
 \left(x, y\in \mathbb{F}\right). 
\]
Due to the additivity of the function $a$, the mapping $B$ is symmetric and biadditive. Consider now the function $\tilde{f}$ defined on $\mathbb{F}$ by 
\[
 \tilde{f}(x)= f(x)-\gamma 
 \qquad 
 \left(x\in \mathbb{F}\right)
\]
to deduce that we have 
\[
 \tilde{f}(x+y)-\tilde{f}(x)-\tilde{f}(y)= B(x, y) 
 \qquad 
 \left(x, y\in \mathbb{F}\right). 
\]
Observe that Proposition \ref{prop_biadditive} can be applied with the choice $S=\mathbb{F}$ to obtain that there exists an additive function $A\colon \mathbb{F}\to \mathbb{F}$ such that 
\[
 \tilde{f}(x)= \frac{1}{2}B(x, x)+A(x)=
 \frac{1}{2} a(x^{2})+A(x) 
 \qquad 
 \left(x\in \mathbb{F}\right).
\]
From this however we conclude that 
\[
 f(x)= \tilde{f}(x)+\gamma= \frac{1}{2} a(x^{2})+A(x)+\gamma
\]
for all $x\in \mathbb{F}$. 
\end{proof}

\begin{rem}
Interestingly, if we replace the inhomogeneity $\alpha(xy)$ in equation \eqref{add_product} with $\alpha(x)\alpha(y)$, the above results are no longer as effectively applicable. 
However, even in this case (by applying the above reasoning based on the cocyle equation), we can state that 
\[
 \alpha(x)\alpha(y)+\alpha(x+y)\alpha(z)= \alpha(x)\alpha(y+z)+\alpha(y)\alpha(z)
\]
holds for all $x, y, z\in \mathbb{F}$. 
From this, with $y=0$ it follows that 
\[
 \alpha(0)\cdot \left[\alpha(x)-\alpha(z)\right]=0 
 \qquad 
 \left(x, z\in \mathbb{F}\right). 
\]
That is, 
$\alpha(0)=0$ or $\alpha$ is constant on $\mathbb{F}$. 
\end{rem}

In what follows, we study the functional equation 
\[
 f(xy)-f(x)-f(y)= \alpha(x)\alpha(y). 
\]
Let us observe that the 'natural' domain of the functions $f$ and $\alpha$ appearing in this equation is a commutative group, while their range is a field. However, the operation defined on the group does not necessarily have to coincide with any of the operations defined on the field. This observation is particularly useful because, when considering real functions, the two equations
\[
 f(xy)-f(x)-f(y)= \alpha(x)\alpha(y)
\]
and 
\[
 f(x+y)-f(x)-f(y)= \alpha(x)\alpha(y)
\]
can be treated as instances of the same case. Indeed, in the first case we take $(G, \cdot)= (\mathbb{R}^{\times}, \cdot)$, while in the second case we put $(G, \cdot)= (\mathbb{R}, +)$.  Accordingly, being an additive function on the group $(\mathbb{R}^{\times}, \cdot)$ means 
\[
 a(x\cdot y)= a(x)+a(y) 
 \qquad 
 \left(x, y\in \mathbb{R}\right), 
\]
while being an additive function on the group $(\mathbb{R}, +)$ means that we have 
\[
 a(x+y)= a(x)+a(y) 
 \qquad 
 \left(x, y\in \mathbb{R}\right). 
\]

As we will see, in the case of the inhomogeneity $\alpha(x)\alpha(y)$, it is worth exploiting the fact that in such cases, the equation in question is a Levi-Civita equation. There is extensive literature on equations of this type; however, these functional equations can only be treated as part of a unified theory if the unknown functions in the equation are complex-valued. Therefore, from this point on, we will only deal with \emph{complex-valued} functions.

\begin{lem}\label{lem_exp_sol}
 Let $(G, \cdot)$ be a group and $f, \alpha \colon G\to \mathbb{C}$ be functions for which 
 \[
  f(x\cdot y)-f(x)-f(y)= \alpha(x)\cdot \alpha(y)
 \]
holds for all $x, y\in G$. Then the function $f$ is an exponential polynomial of degree at most $3$. If, in addition, the system $\{ f, \mathbf{1}, \alpha\}$ is linearly independent, then the function $\alpha$ is also an exponential polynomial of degree at most three.
\end{lem}

\begin{proof}
 The equation appearing in the lemma, after some rearrangement, is the following 
 \[
  f(x\cdot y)=f(x)+f(y)+\alpha(x)\cdot \alpha(y) 
  \qquad 
  \left(x, y\in G\right), 
 \]
or, in a more detalied form 
\[
  f(x\cdot y)=f(x)\cdot \mathbf{1}+\mathbf{1}\cdot f(y)+\alpha(x)\cdot \alpha(y) 
  \qquad 
  \left(x, y\in G\right). 
\]
Observe that this already expresses that for all $y\in G$ we have 
\[
 \tau_{y}f \in \mathrm{lin}(\{f, \mathbf{1}, \alpha \}), 
\]
that is, all the translates of the function $f$ belong to the linear space spanned by the functions $f, \mathbf{1}$ and $\alpha$. So $f$ is an exponential polynomial of degree at most $3$. If, moreover the system  $\{ f, \mathbf{1}, \alpha\}$ is linearly independent, then the same holds true for the function $\alpha$, as well. 
\end{proof}

\begin{thm}\label{thm_exp_sol}
 Let $(G, \cdot)$ be a group and $f, \alpha \colon G\to \mathbb{C}$ be functions for which 
 \begin{equation}\label{eq_exp_sol}
  f(x\cdot y)-f(x)-f(y)= \alpha(x)\cdot \alpha(y)
 \end{equation}
 holds for all $x, y\in G$. Then there exist an additive function $a\colon G\to \mathbb{C}$, an exponential function $m\colon G\to \mathbb{C}$ and complex constants 
 $\alpha$ and $\gamma$ such that 
 \begin{equation}\label{eq_exp_sol_desc}
 \begin{array}{rcl}
 \alpha(x)&=& \alpha\cdot (m(x)-1)\\
  f(x)&=& \gamma \cdot a(x)+\alpha^{2}\cdot (m(x)-1)
  \end{array}
   \end{equation}
 hold for all $x\in G$. And also conversely, if we define the functions $f, \alpha \colon G\to \mathbb{C}$ through \eqref{eq_exp_sol_desc}, then they solve equation \eqref{eq_exp_sol}, provided that $a\colon G\to \mathbb{C}$ is an additive, $m\colon G\to \mathbb{C}$ is an exponential function and $\alpha$ and $\gamma$ are complex constants.  
\end{thm}

\begin{proof}
Depending on the rank of the system $\{f, \mathbf{1}, \alpha\}$, we need to distinguish between several cases. Clearly $\mathrm{rank}(\{f, \mathbf{1}, \alpha\})\in \{ 1, 2, 3\}$.

If $\mathrm{rank}(\{f, \mathbf{1}, \alpha\})=3$, then we have three subcases, depending on whether the variety of the function $f$ contains one, two, or three linearly independent exponential functions. Note that condition 
$\tau_{y}f\in \mathrm{lin}(\{f, \mathbf{1}, \alpha\})$ for all $y\in G$ guarantees that the identically one function (considered as an exponential) is always present in the variety of $f$.
\begin{enumerate}[{Subcase} (i)]
\item In the first subcase, we examine the situation when there is only a single exponential in the variety of the function $f$. Because of the above, this exponential must necessarily be the identically one function. Therefore,
\[
 \alpha(x)= \sum_{p=1}^{2}\alpha_{p}a_{p}(x)+\alpha 
 \quad 
 \text{and}
 \quad 
 f(x)= \sum_{p, q=1}^{2}\delta_{p}\varepsilon_{q} a_{p}(x)a_{q}(x)+\sum_{p=1}^{2}\gamma_{p}a_{p}(x)+\gamma
\]
holds for all $x\in G$ with some constants $\alpha_{p}, \delta_{p}, \varepsilon_{p}, \gamma_{p}, \alpha, \gamma$, and with some linearly independent additive functions $a_{p}\colon G\to \mathbb{C}$ for $p\in \{ 1, 2\}$. 
Substituting these forms back into our equation, 
\begin{align*}
0= & -\gamma - \alpha^2 - \alpha_2 a_2(y) \alpha - \alpha_1 a_1(y) \alpha - \alpha_2 a_2(x) \alpha - \alpha_1 a_1(x) \alpha \\
& + \delta_2 \epsilon_2 a_2(x) a_2(y)^2 - \delta_2 \epsilon_2 a_2(y)^2 + \delta_1 \epsilon_2 a_2(x) a_1(y) a_2(y) \\
& + \epsilon_1 \delta_2 a_1(x) a_1(y) a_2(y) - \delta_1 \epsilon_2 a_1(y) a_2(y) - \epsilon_1 \delta_2 a_1(y) a_2(y) \\
& + \delta_2 \epsilon_2 a_2(x)^2 a_2(y) + \delta_1 \epsilon_2 a_1(x) a_2(x) a_2(y) - \alpha_2^2 a_2(x) a_2(y) \\
& - \alpha_1 \alpha_2 a_1(x) a_2(y) - \gamma_2 a_2(y) + \delta_1 \epsilon_1 a_1(x) a_1(y)^2 - \delta_1 \epsilon_1 a_1(y)^2 \\
& + \epsilon_1 \delta_2 a_1(x) a_2(x) a_1(y) - \alpha_1 \alpha_2 a_2(x) a_1(y) + \delta_1 \epsilon_1 a_1(x)^2 a_1(y) \\
& - \alpha_1^2 a_1(x) a_1(y) - \gamma_1 a_1(y) - \delta_2 \epsilon_2 a_2(x)^2 - \delta_1 \epsilon_2 a_1(x) a_2(x) \\
& - \epsilon_1 \delta_2 a_1(x) a_2(x) - \delta_1 \epsilon_1 a_1(x)^2 
\end{align*}
follows for all $x, y\in G$. As $a_{1}$ and $a_{2}$ are linearly independent additive functions, this is possible only if all the coefficients vanish in the above identity. Especially, the coefficients of $a_{1}(x)a_{1}(y)$ and $a_{2}(x)a_{2}(y)$, that is, $- \alpha_{1}^2$ and $- \alpha_{2}^2$ should equal zero. In this case however, $\alpha_{1}=0$, $\alpha_{2}=0$ and thus $\alpha(x)=\alpha$ follows for all $x\in G$, but then $\mathrm{rank}(\{f, \mathbf{1}, \alpha\})=2$, which is a contradiction. This means that the equation has no solution for which the rank condition would be satisfied. This means that the equation has no solution for which the functions $f$ and $\alpha$ are of the above form. 

\item The second subcase is when there are exactly two exponentials in the variety of $f$, one of which is the identically one function.
\[
 \alpha(x)= \alpha_{2}m(x)+\alpha_{3} 
 \quad 
 \text{and}
 \quad 
 f(x)= \gamma_{1}a(x)+\gamma_{2}m(x)+\gamma_{3} 
\]
holds for all $x\in G$ with some complex constants $\alpha_{2}, \alpha_{3}, \gamma_{1}, \gamma_{2}, \gamma_{3}$, with an additive function $a\colon G\to \mathbb{C}$ and with an exponential $m\colon G\to \mathbb{C}$. Using these representations and our equation, 
\begin{multline*}
(\gamma_2- \alpha_2^2) \cdot m(x) m(y)
+
(- \alpha_2 \alpha_3 - \gamma_2) \cdot m(y)
\\
+(- \alpha_2 \alpha_3 - \gamma_2) m(x)
+(- \gamma_3 - \alpha_3^2)
=0
\end{multline*}
follows for all $x, y\in G$. In view of Lemma \ref{L_lin_dep}, this is possible only if 
\begin{align*}
 \gamma_2- \alpha_2^2 &=0\\
 \alpha_2 \alpha_3 + \gamma_2&=0\\
 \gamma_3  + \alpha_3^2&=0
\end{align*}
holds. Using the first two equations we get that either $\alpha_{2}=0$ or $\alpha_{3}+\alpha_{2}=0$. If $\alpha_{2}=0$, then necessarily $\gamma_{2}=0$ and by the third equation $\gamma_{3}=-\alpha_{3}^{2}$. Thus 
\[
 \alpha(x)= \alpha_{3} 
 \quad 
 \text{and}
 \quad 
 f(x)= \gamma_{1}a(x)-\alpha_{3}^{2} 
 \qquad 
 \left(x\in G\right). 
\]
In this case however $\mathrm{rank}(\{\mathbf{1}, \alpha\})=1$, so $\mathrm{rank}(\{f, \mathbf{1}, \alpha\})<3$ which is not possible by our assumptions. 

The other possibility is that $\alpha_{3}+\alpha_{2}=0$, and then $\gamma_{2}= \alpha_{2}^{2}$ and $\gamma_{3}=-\alpha_{2}^2$. Thus 
\[
 \alpha(x)= \alpha_{2}m(x)-\alpha_{2} 
 \quad 
 \text{and}
 \quad 
 f(x)= \gamma_{1}a(x)+\alpha_{2}^{2}m(x)-\alpha_{2}^{2}
\]
holds for all $x\in G$. An easy computation shows that these functions indeed solve our equation.  
\item The third case would be that the variety of the function $f$ contains three linearly independent exponential functions, but according to Székelyhidi \cite[page 92]{Sze91}, this is impossible.
\end{enumerate}

The second possibility is that $\mathrm{rank}(\{f, \mathbf{1}, \alpha \})=2$. 

In this case, our equation says that all the translates of the function $f$ belong to the two-dimensional linear space $\mathrm{lin}\{f, \mathbf{1}, \alpha \}$. Thus $f$ is an exponential polynomial of degree at most two. 
Again, depending on whether the variety of the function $f$ contains one,  or two linearly independent exponential functions, we have to distinguish subcases. Note that condition 
$\tau_{y}f\in \mathrm{lin}(\{f, \mathbf{1}, \alpha\})$ for all $y\in G$ guarantees that the identically one function (considered as an exponential) is always present in the variety of $f$.

Further, the assumption that $\mathrm{rank}(\{f, \mathbf{1}, \alpha \})=2$ yields that the system $\{f, \mathbf{1}, \alpha \}$ is linearly dependent. So there exist complex numbers $\lambda_{1}, \lambda_{2}$ and $\lambda_{3}$ that are not simoultaneously zero, such that 
\[
 \lambda_{1}f(x)+ \lambda_{2}\alpha(x)+\lambda_{3}=0
\]
holds for all $x\in G$. Observe that if $\lambda_{2}=0$ would hold, then the function $f$ would be  constant. Then from the functional equation $\alpha\equiv \mathrm{const.}$ would follow. At the same time, this would imply that $\mathrm{rank}(\{f, \mathbf{1}, \alpha \})=1$, which is not possible in this case. So $\lambda_{2}\neq 0$. In this case however, 
\[
 \alpha(x)= -\frac{\lambda_{1}}{\lambda_{2}}f(x)-\frac{\lambda_{3}}{\lambda_{2}} 
 \qquad 
 \left(x\in G\right)
\]
follows. Thus the function $\alpha$ is a linear combination of the functions $f$ and $\mathbf{1}$, i.e., 
\[
 \alpha(x)= \lambda f(x)+\mu. 
\]
Since the function $f$ is an exponential polynomial of degree at most two, we need to distinguish two cases, depending on whether the variety of the function $f$ contains one,  or two linearly independent exponential functions. 
\begin{enumerate}[{Subcase} (i)]
 \item The first possibility is that 
 \[
  f(x)= (\alpha_{1}a(x)+\alpha_{2})m(x) 
  \qquad 
  \left(x\in G\right), 
 \]
where $a\colon G\to \mathbb{C}$ is an additive function, $m\colon G\to \mathbb{C}$ is an exponential, and $\alpha_{1}, \alpha_{2}$ are complex constants. Accordingly, for the function $\alpha$, 
\begin{align*}
 \alpha(x)&= \lambda f(x)+\mu\\
 &= \lambda (\alpha_{1}a(x)+\alpha_{2})m(x)+\mu 
 \qquad 
 \left(x\in G\right). 
\end{align*}
Inserting these forms back to the functional equation, and using the linear independence of the involved functions, we obtain $\alpha_{1}^2\lambda^{2}=0$. If $\alpha_{1}=0$ would hold, then, again, from the functional equation itself, $\mu=0$ and $\alpha_{2}=0$ would follow. This would however mean that $f\equiv 0$, which is impossible. The other possibility is that $\lambda=0$, so $\alpha(x)= \mu\; (x\in G)$. Inserting this into the functional equation, 
\[
 f(xy)= f(x)+f(y)+\mu^{2} 
 \qquad 
 \left(x, y\in G\right).
\]
Therefore, the function $f+\mu^{2}$ is additive. So we have 
\[
f(x)= a(x)-\mu^{2} 
\qquad 
\left(x\in G\right). 
\]
Let us observe that this form is consistent with the above representation, if we choose the exponential $m\equiv 1$.
\item The other possibility is that there exist linearly independent exponentials $m_{1}, \allowbreak m_{2}\colon G\to \mathbb{C}$ and complex constants $\alpha_{1}, \alpha_{2}$ such that 
\[
 f(x)= \alpha_{1}m_{1}(x)+\alpha_{2}m_{2}(x) 
 \qquad 
 \left(x\in G\right). 
\]
Let us observe that one of the two exponential functions must be the identically one function, since otherwise $\operatorname{rank}(f, \alpha, \mathbf{1}) \geq \operatorname{rank}(m_1, m_2, {1}) = 3$, which is impossible because $\operatorname{rank}(f, \alpha, {1}) = 2$. Therefore, without the loss of generality $m_{2}\equiv 1$ can be assumed. 
This means that 
\begin{align*}
 f(x)&= \alpha_{1}m_{1}(x)+\alpha_{2}\\
 \alpha(x)& = \lambda f(x)+\mu= \lambda (\alpha_{1}m_{1}(x)+\alpha_{2}) 
 \qquad 
 \left(x\in G\right). 
\end{align*}
Inserting these representations back to our functional equation, the system of equations 
\begin{align*}
 \alpha_{1}&= \alpha_{1}^{2}\lambda^{2}\\
 0&=\alpha_{1}+\alpha_{1}\lambda\mu +\alpha_{1}\alpha_{2}\lambda^{2}\\
 0&= \alpha_{2}+\mu^{2}+2\alpha_{2}\lambda\mu +\alpha_{2}^2\lambda^{2}
\end{align*}
follows, as the involved functions are linearly independent. 
This system of equations has two different types of solutions. In the first case, $\alpha_{1} = 0$. However, this would mean that the function $f$ is constant, which is impossible due to the rank condition. In the second case,
$ \alpha_{1}= \dfrac{1}{\lambda^2}$, $\alpha_{2}= -\dfrac{1}{\lambda^2}$, $\lambda$ is arbitrary, and  $\mu=0$. For the functions $f, \alpha$ this yields that 
\[
 f(x)= \frac{1}{\lambda^2}(m(x)-1) 
 \qquad 
 \text{and}
 \qquad 
 \alpha(x)= \frac{1}{\lambda}(m(x)-1) 
 \qquad 
 \left(x\in G\right). 
\]
An easy computation shows that these functions indeed solve the functional equation. 
\end{enumerate}

Finally, the last case is that $\mathrm{rank}(\{f, \mathbf{1}, \alpha \})=1$. In this case, both the functions $f$ and $\alpha$ are constant functions. In other words, there exist complex constants $\alpha$ and $\gamma$ such that 
\[
 \alpha(x)= \alpha 
 \quad 
 \text{and}
 \quad
 f(x)= \gamma 
 \qquad 
 \left(x\in G\right). 
\]
Substituting back into our equation, we conclude that these constants should fulfill 
\[
 -\gamma= \alpha^2. 
\]
Thus 
\[
 \alpha(x)= \alpha 
 \quad 
 \text{and}
 \quad
 f(x)= -\alpha^2 
 \qquad 
 \left(x\in G\right)
\]
with a complex constant $\alpha$. 
\end{proof}

If we take $(G, \cdot)= (\mathbb{R}, +)$ in Theorem \ref{thm_exp_sol}, then the corollary below follows immediately. 

\begin{cor}[Solutions of the 'additive' equation on $\mathbb{R}$]
 The functions  $f, \alpha \colon \mathbb{R}\to \mathbb{C}$ satisfy functional equation
 \begin{equation}
  f(x+y)-f(x)-f(y)= \alpha(x)\alpha(y)
 \end{equation}
for all $x, y\in \mathbb{R}$, if and only if there exist an additive function $a\colon \mathbb{R}\to \mathbb{C}$, an exponential function $m\colon \mathbb{R}\to \mathbb{C}$ and complex constants 
 $\alpha$ and $\gamma$ such that 
 \begin{equation}
 \begin{array}{rcl}
 \alpha(x)&=& \alpha (m(x)-1)\\
  f(x)&=& \gamma  a(x)+\alpha^{2} (m(x)-1)
  \end{array}
   \end{equation}
 hold for all $x\in \mathbb{R}$. 
 \end{cor}

\begin{cor}[Regular solutions of the 'additive' equation on $\mathbb{R}$]\label{cor_add_reg}
 Let us assume that the functions $f, \alpha \colon \mathbb{R}\to \mathbb{R}$ fulfill at  one of the following conditions:
 \begin{enumerate}[(i)]
 \item $f$ and $\alpha$ are continuous on $\mathbb{R}$, 
 \item $f$ and $\alpha$ are continuous at a point $x_{0}\in \mathbb{R}$, 
 \item $f$ and $\alpha$ are measurable on a set of positive measure; 
 \item $f$ and $\alpha$ are bounded on a nonempty open subset of $\mathbb{R}$. 
\end{enumerate} 
Then $f$ and $\alpha$ satisfy the functional equation 
 \begin{equation}
  f(x+y)-f(x)-f(y)= \alpha(x) \alpha(y)
 \end{equation}
for all $x, y\in \mathbb{R}$, if and only if, 
there exist real constants $\alpha, \gamma$  and $\lambda$ such that 
  \[
 \begin{array}{rcl}
 \alpha(x)&=& \alpha \cdot (\exp(\lambda x )-1)\\
  f(x)&=& \gamma   x+\alpha^{2}\cdot (\exp(\lambda x )- 1)
  \end{array}
   \]
 hold for all $x\in \mathbb{R}$, or 
 \[
  \begin{array}{rcl}
 \alpha(x)&=& -\alpha\\
  f(x)&=& \gamma  x- \alpha^{2}
  \end{array}
  \]
for all $x\in \mathbb{R}$. 
\end{cor}

\begin{proof}
Using Theorem \ref{thm_exp_sol}, we immediately get that the function $\alpha$ is of the form 
\[
 \alpha(x)= \alpha\cdot  (m(x)-1)  
 \qquad 
 \left(x\in \mathbb{R}\right). 
\]
If $\alpha=0$, then the function $\alpha$ is the identically zero function. In this case , $f$ is an additive function that fulfills at least one of the conditions (i)--(iv). So $f\colon \mathbb{R}\to \mathbb{R}$ is a regular additive function, i.e., 
\[
 f(x)= \gamma x 
\]
holds for all $x\in \mathbb{R}$ with some real constant $\gamma$. 

If $\alpha \neq 0$, then we have 
\[
 m(x)= \frac{\alpha(x)}{\alpha}+1 
 \qquad 
 \left(x\in \mathbb{R}\right). 
\]
As the function $\alpha$ fulfills at least one of the conditions (i)--(iv) appearing in the statement, the same holds for the exponential $m$. Thus $m\colon \mathbb{R}\to \mathbb{R}$ is a regular exponential function, i.e., either $m$ is identically zero, or there exists a real number $\lambda$ such that $m(x)= \exp(\lambda x)\; (x\in \mathbb{R})$. 
\end{proof}

\begin{rem}[One-sided boundedness]
 Let us observe that the fact that the functions $f$ and $\alpha$ are bounded from \emph{below} on $\mathbb{R}$ does not necessarily imply that they can be represented as described in the previous corollary. Indeed, any exponential function $m\colon \mathbb{R}\to \mathbb{R}$ is bounded below on $\mathbb{R}$. Accordingly, if take a nonnegative real number $\alpha$ and consider the functions 
 \[
  \begin{array}{rcl}
 \alpha(x)&=& \alpha (m (x)-1)\\
  f(x)&=& \gamma   x+\alpha^{2} (m (x)- 1)
  \end{array}
  \qquad 
  \left(x\in \mathbb{R}\right), 
 \]
then they solve the functional equation, but they are not necessarily of the form as presented in Corollary \ref{cor_add_reg}. 

The same holds for those solutions $f, \alpha\colon \mathbb{R}\to \mathbb{R}$ that are bounded only from \emph{above}. In such cases, the above reasoning must be applied with nonpositive $\alpha$.
\end{rem}

In what follows let $D\in \left\{\mathbb{R}^{\times}, ]0, +\infty[ \right\}$. As $(D, \cdot)$ is a commutative group, in view of Theorem \ref{thm_exp_sol}, we obtain the following. 

\begin{cor}[Solutions of the 'multiplicative' equation on $D$]\label{cor_multD}
 Let  $f, \alpha \colon D \to \mathbb{C}$ be functions for which 
 \begin{equation}\label{eq_mult}
  f(xy)-f(x)-f(y)= \alpha(x)\cdot \alpha(y)
 \end{equation}
 holds for all $x, y\in D$. Then there exist a logarithmic function $\ell \colon D \to \mathbb{C}$, a multiplicative function $m\colon \mathbb{R}\to \mathbb{C}$ and complex constants 
 $\alpha$ and $\gamma$ such that 
 \begin{equation}\label{eq_mult_sol}
 \begin{array}{rcl}
 \alpha(x)&=& \alpha (m(x)-1)\\
  f(x)&=& \gamma  \ell (x)+\alpha^{2}\cdot (m(x)-1)
  \end{array}
   \end{equation}
 hold for all $x\in D$. And also conversely, if we define the functions $f, \alpha \colon D\to \mathbb{C}$ through the formulas \eqref{eq_mult_sol}, then they solve the functional  equation above, provided that $\ell \colon \mathbb{R}^{\times}\to \mathbb{C}$ is a logarithmic, $m\colon \mathbb{R}^{\times}\to \mathbb{C}$ is a multiplicative function and $\alpha$ and $\gamma$ are complex constants.  
\end{cor}

\begin{cor}[Regular solutions of the 'multiplicative' equation on $D$]
 Let  $f, \alpha \colon D \to \mathbb{R}$ be functions for which 
 \begin{equation}\label{eq_mult_reg}
  f(xy)-f(x)-f(y)= \alpha(x)\cdot \alpha(y)
 \end{equation}
 holds for all $x, y\in D$.
 Let us further assume that at least one of the following conditions is satisfied.
\begin{enumerate}[(i)]
 \item $f$ and $\alpha$ are continuous on $D$, 
 \item $f$ and $\alpha$ are continuous at a point $x_{0}\in D$, 
 \item $f$ and $\alpha$ are measurable on a set $\tilde{D}\subset D$ of positive measure; 
 \item $f$ and $\alpha$ are bounded on a nonempty open subset of $D$. 
\end{enumerate}  
 Then there exist real numbers  $\alpha$ and $\gamma$ such that 
 \begin{equation}\label{eq_mult_sol_reg}
 \begin{array}{rcl}
 \alpha(x)&=& \alpha\cdot (m(x)-1)\\
  f(x)&=& \gamma \cdot \ln(|x|)+\alpha^{2}\cdot (m(x)-1)
  \end{array}
   \end{equation}
 hold for all $x\in D$, where the function $m\colon D\to \mathbb{R}$ has one of the following forms
 \begin{align*}
m(x)&=0 & (x\in D)\\
m(x)&= |x|^{\mu} & (x\in D)\\
m(x)&= \operatorname{sign}(x) |x|^{\mu} & (x\in D)
 \end{align*}
 with an appropriate real number $\mu$. 
 And also conversely, if we define the functions $f, \alpha \colon D\to \mathbb{C}$ through the formulas \eqref{eq_mult_sol_reg}, then they solve the functional  equation above, provided that $m\colon D\to \mathbb{C}$ is one of the above functions,  and $\alpha$ and $\gamma$ are real constants.  
\end{cor}

The functional equation under investigation also makes sense for functions $f$ and $\alpha$ whose domain includes $0$. As we will demonstrate below, in such cases the functional equation has significantly fewer solutions. Accordingly, let 
 $D^{\ast}\in \left\{\mathbb{R}, [0, +\infty[\right\}$. 
 If $D^{\ast}=\mathbb{R}$, then $D^{\ast}\setminus \left\{ 0\right\} =\mathbb{R}^{\times}$. 
 Similarly, if $D^{\ast}=[0, +\infty[$, then $D^{\ast}\setminus \left\{ 0\right\} =]0, +\infty[$.

\begin{cor}[Solutions of the 'multiplicative' equation on $D^{\ast}$]
 Let  $f, \alpha \colon D^{\ast} \to \mathbb{C}$ be functions for which 
 \[
  f(xy)-f(x)-f(y)= \alpha(x)\cdot \alpha(y)
 \]
 holds for all $x, y\in D^{\ast}$. Then there exist a multiplicative function $m\colon \mathbb{R}\to \mathbb{C}$ and complex constants 
 $\alpha$ and $\gamma$ such that 
 \[
 \begin{array}{rcl}
 \alpha(x)&=& \alpha (m(x)-1)\\
  f(x)&=& \alpha^{2} (m(x)-1)
  \end{array}
   \]
 hold for all $x\in D$. And also conversely, if we define the functions $f, \alpha \colon G\to \mathbb{C}$ through the above formulas, then they solve the functional  equation above, provided that $m\colon \mathbb{R}\to \mathbb{C}$ is a multiplicative function and $\alpha$ and $\gamma$ are complex constants.  
\end{cor}

\begin{proof}
 As $D^{\ast}\setminus \left\{ 0\right\} \in \left\{\mathbb{R}^{\times}, ]0, +\infty [\right\}$, with the application of Corollary \ref{cor_multD}, we obtain that there exist a logarithmic function $\ell \colon \mathbb{R}\to \mathbb{C}$, a multiplicative function $m\colon \mathbb{R}\to \mathbb{C}$ and complex constants 
 $\alpha$ and $\gamma$ such that 
 \begin{equation}
 \begin{array}{rcl}
 \alpha(x)&=& \alpha\cdot (m(x)-1)\\
  f(x)&=& \gamma \cdot \ell (x)+\alpha^{2}\cdot (m(x)-1)
  \end{array}
   \end{equation}
 hold for all $x\in D$. Let now $x\in D$ be arbitrary and $y=0$ in our functional equation, to get that
 \[
  f(x)= -\alpha(0)\alpha(x)
 \]
should hold for all $x\in D$. This however means that 
\[
 \gamma  \ell (x)+\alpha^{2} \cdot (m(x)-1)
 = -\alpha(0) \alpha \cdot (m(x)-1), 
\]
that is, 
\[
 \gamma  \ell (x)+(\alpha^{2}+\alpha  \alpha(0)) (m(x)-1)
 = 0
\]
holds for all $x\in D$. From which $\gamma=0$ follows. 
\end{proof}

With the aid of the previous corollary, the regular solutions can be determined easily. 

\begin{cor}[Regular solutions of the 'multiplicative' equation on $D^{\ast}$]
 Let  $f, \alpha \colon D^{\ast} \to \mathbb{C}$ be functions for which 
 \[
  f(xy)-f(x)-f(y)= \alpha(x)\alpha(y)
 \]
 holds for all $x, y\in D^{\ast}$. Then there exist a multiplicative function $m\colon \mathbb{R}\to \mathbb{C}$ and complex constants 
 $\alpha$ and $\gamma$ such that 
 \[
 \begin{array}{rcl}
 \alpha(x)&=& \alpha\cdot (m(x)-1)\\
  f(x)&=& \alpha^{2}\cdot (m(x)-1)
  \end{array}
   \]
 hold for all $x\in D^{\ast}$, where the function $m\colon D^{\ast}\to \mathbb{R}$ has one of the following forms
 \begin{align*}
m(x)&=0 & (x\in D)\\
m(x)&= |x|^{\mu} & (x\in D)\\
m(x)&= \operatorname{sign}(x) |x|^{\mu} & (x\in D)
 \end{align*}
 with an appropriate nonnegative real number $\mu$.   And also conversely, if we define the functions $f, \alpha \colon G\to \mathbb{C}$ through the above formulas, then they solve the functional  equation above, provided that $m\colon \mathbb{R}\to \mathbb{C}$ has one of the above forms, and $\alpha$ is a real constant.  
\end{cor}

\subsection{Exponential Cauchy differences}

\begin{prop}\label{prop4}
 Let $(G, \cdot)$ be a semigroup and $(\mathbb{F}, +, \cdot)$ be a field. 
 If the functions $f, \alpha \colon G\to \mathbb{F}$ satisfy 
 \begin{equation}\label{exp_product}
 f(xy)-f(x)f(y)= \alpha(xy)  
 \end{equation}
for all $x, y\in G$, then there exists an exponential function $m\colon G\to \mathbb{F}$ such that 
\[
 f(x)= f(1)\cdot m(x)
 \quad 
 \text{and}
 \quad 
 \alpha(x)= f(1)(1-f(1))\cdot m(x)
\]
for all $x\in G$. 
\end{prop}

\begin{proof}
 Consider the function $\tilde{f}\colon G\to \mathbb{F}$ defined by 
 \[
  \tilde{f}(x)= f(x)- \alpha(x) 
  \qquad 
  \left(x\in G\right)
 \]
to deduce that equation \eqref{exp_product} can be written as 
\[
 \tilde{f}(xy)= f(x)f(y) 
 \qquad 
 \left(x, y\in G\right)
\]
With $y=1$, 
\[
 \tilde{f}(x)= f(1)f(x)
\]
follows for all $x\in G$. Therefore 
\[
 f(1)f(xy)= f(x)f(y)
\]
holds for all $x, y\in G$. If $f(1)=0$, then
\[
 f(x)f(y)=0
\]
should hold for all $x, y\in G$. So $f$ is identically zero. If however, $f(1)\neq 0$, then the above identity implies that 
\[
 \frac{f(xy)}{f(1)}= \frac{f(x)}{f(1)}\cdot \frac{f(y)}{f(1)}
\]
holds for all $x, y\in G$. In other words, there exists an exponential function $m\colon G\to \mathbb{F}$ such that 
\[
 f(x)= f(1)m(x) 
 \qquad 
 \left(x\in G\right). 
\]
Note that this representation is also valid when $f(1)=0$. 

Further, from one hand 
\[
 \tilde{f}(x)= f(1)f(x)= f(1)(f(1)m(x))= f(1)^{2}m(x) 
 \qquad 
 \left(x\in G\right). 
\]
On the other hand 
\[
 \tilde{f}(x)= f(x)-\alpha(x)= f(1)m(x)-\alpha(x)
 \qquad 
 \left(x\in G\right). 
\]
Therefore 
\[
 f(1)m(x)-\alpha(x)= f(1)^{2}m(x) 
\]
should hold for all $x\in G$, from which 
\[
 \alpha(x)= (f(1)-f(1)^{2})m(x)
\]
follows for all $x\in G$. 
\end{proof}

In the corollaries below let $\mathbb{K}\in \left\{ \mathbb{R}, \mathbb{C}\right\}$. 

\begin{cor}\label{cor11}
  If the functions $f, \alpha \colon \mathbb{K}^{\times}\to \mathbb{K}$ satisfy 
 \begin{equation}\label{eq_cor11}
 f(xy)-f(x)f(y)= \alpha(xy)  
 \end{equation}
for all $x, y\in \mathbb{K}^{\times}$, then there exists an exponential function $m\colon \mathbb{K}^{\times}\to \mathbb{F}$, i.e., a function for which we have $m(xy)=m(x)m(y)$ for all $x, y\in \mathbb{K}^{\times}$ such that 
\[
 f(x)= f(1)\cdot m(x)
 \quad 
 \text{and}
 \quad 
 \alpha(x)= f(1)(1-f(1))\cdot m(x)
\]
for all $x\in \mathbb{K}^{\times}$. 
\end{cor}

\begin{cor}\label{cor12}
  If the functions $f, \alpha \colon \mathbb{K}\to \mathbb{K}$ satisfy 
 \begin{equation}\label{eq_cor12}
 f(x+y)-f(x)f(y)= \alpha(x+y)  
 \end{equation}
for all $x, y\in \mathbb{K}$, then there exists an exponential function $m\colon \mathbb{K}\to \mathbb{F}$, i.e., a function for which we have $m(x+y)=m(x)m(y)$ for all $x, y\in \mathbb{K}$ such that 
\[
 f(x)= f(1)\cdot m(x)
 \quad 
 \text{and}
 \quad 
 \alpha(x)= f(1)(1-f(1))\cdot m(x)
\]
for all $x\in \mathbb{K}$. 
\end{cor}

\begin{rem}
 If $\mathbb{K}= \mathbb{R}$ then with the help of the above two corollaries, we can immediately determine the regular solutions of equations \eqref{eq_cor11} and \eqref{eq_cor12} as well. 
\end{rem}

Finally, in this part of the manuscipt, we intend to study the equation
\[
 f(x\cdot y)-f(x)f(y)= \alpha(x)\alpha(y)  
\]
for functions $f, \alpha\colon G\to \mathbb{C}$ defined on a commutative group. In this case, it is again worth taking advantage of the fact that, after rearrangement, the equation can be transformed into a Levi-Civita type equation of the form 
\[
 f(x\cdot y)= f(x)f(y)+\alpha(x)\alpha(y). 
\]

\begin{thm}\label{thm_exp_difference_prod}
 Let $(G, \cdot)$ be a group and $f, \alpha \colon G\to \mathbb{C}$ be functions for which 
 \begin{equation}\label{exp_difference_prod}
  f(x\cdot y)= f(x)f(y)+\alpha(x)\alpha(y) 
 \end{equation}
 holds for all $x, y\in G$. Then there exist and additive function $a\colon G\to \mathbb{C}$, exponential functions $m, m_{1}, m_{2}\colon G\to \mathbb{C}$ such that the system $\{ m_{1}, m_{2} \}$ is linearly independent, further, there are complex constants $\alpha_{1}, \alpha_{2}, \gamma_{1}, \gamma_{2}$ and $\gamma\neq 0$, such that the pair $(f, \alpha)$ has one of the following forms. 
 \begin{enumerate}[(i)]
  \item 
  \begin{align*}
   f(x)&= (\gamma_{1} a(x)+\gamma_{2})\cdot m(x)\\
   \alpha(x)&= (\alpha_{1} a(x)+\alpha_{2})\cdot m(x)
   \qquad (x\in G), 
  \end{align*}
where the constants $\alpha_{1}, \alpha_2, \gamma_{1}$ and $\gamma_{2}$ satisfy the system of equation
\begin{align*}
 \gamma_{1}^{2}+\alpha_{1}^2&=0\\
 \gamma_{2}^{2}+\alpha_{2}^2&=\gamma_{2}\\
 \gamma_{1}\gamma_{2}+\alpha_{1}\alpha_{2}&= \gamma_{1}. 
\end{align*}
\item \begin{align*}
   f(x)&= \gamma_{1}\cdot m_{1}(x)+\gamma_{2}\cdot m_{2}(x)\\
   \alpha(x)&= \alpha_{1}\cdot m_{1}(x)+\alpha_{2}\cdot m_{2}(x)
   \qquad (x\in G), 
  \end{align*}
where the constants $\alpha_{1}, \alpha_2, \gamma_{1}$ and $\gamma_{2}$ satisfy the system of equation
\begin{align*}
 \gamma_{1}^{2}+\alpha_{1}^2&=\gamma_{1}\\
 \gamma_{2}^{2}+\alpha_{2}^2&=\gamma_{2}\\
 \gamma_{1}\gamma_{2}+\alpha_{1}\alpha_{2}&= 0. 
 \end{align*}
 \item either $f$ and $\alpha$ are identically zero, or 
  \begin{align*}
 f(x)&= \gamma \cdot m(x)\\
 \alpha(x)&= \pm\sqrt{\gamma(1-\gamma)}\cdot m(x)
 \qquad 
 \left(x\in G\right). 
 \end{align*}
 \end{enumerate}
\end{thm}

\begin{proof}
 Observe that equation \eqref{exp_difference_prod} expresses that all the translates of the function 
 $f\colon G\to \mathbb{C}$ are contained in the linear space $\mathrm{lin}(\left\{ f, \alpha\right\})$. 
 
 If $\mathrm{rank}(\left\{ f, \alpha\right\})=2$, then this yields that $f$ and $\alpha$ are exponential polynomials of degree one. Thus they are either of the form 
 \begin{align*}
  f(x)&= (\gamma_{1}a(x)+\gamma_{2})\cdot m(x)\\
  \alpha(x)&= (\alpha_{1}a(x)+\alpha_{2})\cdot m(x)
  \qquad 
  \left(x\in G\right), 
 \end{align*}
 or they are of the form 
 \begin{align*}
  f(x)&= \gamma_{1}m_{1}(x)+\gamma_{2}m_{2}(x)\\
  \alpha(x)&= \alpha_{1}m_{1}(x)+\alpha_{2}m_{2}(x)
  \qquad 
  \left(x\in G\right), 
 \end{align*}
where $a\colon G\to \mathbb{C}$ is an additive function, $m, m_{1}, m_{2}\colon G\to \mathbb{C}$ are exponentials so that $\{ m_{1}, m_{2}\}$ is linearly independent and $\alpha_{1}, \alpha_{2}, \gamma_{1}, \gamma_{2}$ are complex constants. Inserting these representations back to equation \eqref{exp_difference_prod} we obtain in the first case 
\begin{align*}
 \gamma_{1}^{2}+\alpha_{1}^2&=0\\
 \gamma_{2}^{2}+\alpha_{2}^2&=\gamma_{2}\\
 \gamma_{1}\gamma_{2}+\alpha_{1}\alpha_{2}&= \gamma_{1}, 
\end{align*} 
and the system of equations 
\begin{align*}
 \gamma_{1}^{2}+\alpha_{1}^2&=\gamma_{1}\\
 \gamma_{2}^{2}+\alpha_{2}^2&=\gamma_{2}\\
 \gamma_{1}\gamma_{2}+\alpha_{1}\alpha_{2}&= 0. 
 \end{align*}
 follows in the second case. These are  cases (i) and (ii) of the statement. 
 
 Finally, if $\mathrm{rank}(\{ f, \alpha \})<2$, then we have 
 \begin{align*}
  f(x)&= \gamma m(x)\\
  \alpha(x)&= \alpha m(x) 
  \qquad 
  \left(x\in G\right), 
 \end{align*}
with an appropriate exponential $m\colon G\to \mathbb{C}$ and complex constants $\alpha$ and $\gamma$. Using these forms, equation \eqref{exp_difference_prod} implies that 
\[
 \gamma = \gamma^2+\alpha^2, 
\]
So $\alpha= \pm\sqrt{\gamma(1-\gamma)}$. 
\end{proof}

 \begin{rem}
 It is sufficient to determine those solutions of the systems of equations appearing in the previous theorem for which there correspond nondegenerate pairs of solutions $(f, \alpha)$. Indeed, if for example the coefficients $\alpha_{1}$ and $\gamma_{1}$ were zero in (i) or in (ii), then the degree of the corresponding exponential polynomials $f$ and $\alpha$ would be one. 
  The nontrivial solutions of the first system, i.e., 
   \begin{align*}
     \gamma_{1}^{2}+\alpha_{1}^2&=0\\
     \gamma_{2}^{2}+\alpha_{2}^2&=\gamma_{2}\\
     \gamma_{1}\gamma_{2}+\alpha_{1}\alpha_{2}&= \gamma_{1}
    \end{align*}
  are 
    \[
     \alpha_{1} \text{ is arbitrary} 
     \qquad 
     \alpha_{2}= 0
     \qquad 
     \gamma_{1}= \pm \alpha_{1}i
     \qquad 
     \gamma_{2}= 1. 
    \]
  Accordingly, the solutions in case (i) are of the form 
 \begin{align*}
  f(x)&= \left(\pm \alpha_{1} i a(x)+1\right)\cdot m(x)\\
  \alpha(x)&= \alpha_{1}a(x)\cdot m(x)  
 \end{align*}
 Similarly, the nontrivial solutions of the second system, i.e., 
\begin{align*}
 \gamma_{1}^{2}+\alpha_{1}^2&=\gamma_{1}\\
 \gamma_{2}^{2}+\alpha_{2}^2&=\gamma_{2}\\
 \gamma_{1}\gamma_{2}+\alpha_{1}\alpha_{2}&= 0. 
 \end{align*}
 are 
 \[
     \alpha_{1} \text{ is arbitrary} 
     \qquad 
     \alpha_{2}= -\alpha_{1}
     \qquad 
     \gamma_{1}= \frac{1\pm \sqrt{1-4\alpha_{1}^2}}{2}
     \qquad 
     \gamma_{2}= \frac{1\mp \sqrt{1-4\alpha_{1}^2}}{2}. 
    \]
 Here we apply the convention that if we choose the $+$ sign for $\gamma_{1}$, then we must choose the $-$ sign for $\gamma_{2}$, and vice versa. 
Accordingly, in case (ii) the solutions are
\begin{align*}
 f(x)&= \frac{1\pm \sqrt{1-4\alpha_{1}^2}}{2} m_{1}(x) + \frac{1\mp \sqrt{1-4\alpha_{1}^2}}{2} m_{2}(x)\\
 \alpha(x)&=\alpha_{1}m_{1}(x)-\alpha_{1}m_{2}(x)\\
 &= \alpha_{1}\cdot (m_{1}(x)-m_{2}(x))
\end{align*}
\end{rem}

\begin{rem}[Real-valued solutions of equation \eqref{exp_difference_prod}]
With the help of Theorem \ref{thm_exp_difference_prod}, we can easily obtain the real-valued solutions of equation \eqref{exp_difference_prod} as well. Note that these may be significantly fewer than the complex-valued solutions. Indeed, the functions 
\begin{align*}
  f(x)&= \left(\pm \alpha_{1} i a(x)+1\right)\cdot m(x)\\
  \alpha(x)&= \alpha_{1}a(x)\cdot m(x)  
 \end{align*}
can both be real-valued only if $\alpha_1 = 0$, that is, if $\alpha$ is identically zero and $f$ is exponential.
Similarly, the functions 
 \begin{align*}
 f(x)&= \frac{1\pm \sqrt{1-4\alpha_{1}^2}}{2} m_{1}(x) + \frac{1\mp \sqrt{1-4\alpha_{1}^2}}{2} m_{2}(x)\\
 \alpha(x)&= \alpha_{1}(m_{1}(x)-m_{2}(x))
\end{align*}
are real-valued if and only if $\alpha_{1}\in \left[-\frac{1}{2}, \frac{1}{2}\right]$. 

Finally, the functions 
\begin{align*}
 f(x)&= \gamma \cdot m(x)\\
 \alpha(x)&= \pm\sqrt{\gamma(1-\gamma)}\cdot m(x)
 \qquad 
 \left(x\in G\right)
 \end{align*}
 are real-valued if and only if $\gamma \in [0, 1]$. 
 
 Summing up, if the functions $f, \alpha \colon G\to \mathbb{R}$ solve equation \eqref{exp_difference_prod}, then they are one of the following form
 \begin{enumerate}[(a)]
  \item \begin{align*}
 f(x)&= \frac{1\pm \sqrt{1-4\alpha^2}}{2} m_{1}(x) + \frac{1\mp \sqrt{1-4\alpha^2}}{2} m_{2}(x)\\
 \alpha(x)&= \alpha(m_{1}(x)-m_{2}(x)) 
 \qquad 
 \left(x\in G\right), 
\end{align*}
where $m_{1}, m_{2}\colon G\to \mathbb{R}$ are exponentials and $\alpha \in \left[-\frac{1}{2}, \frac{1}{2}\right]$, 
\item   \begin{align*}
 f(x)&= \gamma \cdot m(x)\\
 \alpha(x)&= \pm\sqrt{\gamma(1-\gamma)}\cdot m(x)
 \qquad 
 \left(x\in G\right)
 \end{align*}
where $\gamma \in [0, 1]$. 
\end{enumerate}
\end{rem}

\begin{opp}
If $\mathbb{F}$ and $\mathbb{K}$ are fields, and $f, \alpha\colon \mathbb{F}\to \mathbb{K}$ are functions, then the equation 
\[
 f(x+y)-f(x)f(y)= \alpha(xy)   
 \qquad 
 \left(x, y\in \mathbb{F}\right)
\]
as well as 
\[
 f(xy)-f(x)f(y)= \alpha(x+y)  
 \qquad 
 \left(x, y\in \mathbb{F}\right)
\]
also makes sense, however, we cannot solve them using the methods presented in the ma\-nu\-script. Their investigation may be addressed in the future, in a separate manuscript, likely employing entirely different methods. 
\end{opp}

For easier overview, we summarize the results of our paper in the table below.

\begin{tabular}{|p{5cm}|c|c|}
\hline 
 Function(s) & Equation & Statement \\ \hline \hline
$f\colon S\to \mathbb{F}, B\colon S\times S\to \mathbb{F}$ symmetric and biadditive  & $f(x+y)-f(x)-f(y)= B(x, y)$  & Proposition \ref{prop_biadditive} \\ \hline
$f\colon S\to \mathbb{F}$ & $f(x+y)-f(x)-f(y)= \alpha xy$& Corollary \ref{cor1} \\ \hline
$f\colon S\to \mathbb{F}$ & $f(xy)-f(x)-f(y)= \alpha xy$& Corollary \ref{lem_exp_difference}\\ \hline
$f\colon \mathbb{R}\to \mathbb{R}$ & $f(x+y)-f(x)-f(y)= \alpha xy$& Corollary \ref{cor_add}\\ \hline
$f\colon D\to \mathbb{R}$ & $f(xy)-f(x)-f(y)= \alpha xy$& Corollary \ref{cor_exp_positive}\\ \hline
$f\colon \mathbb{R}\to \mathbb{R}$ & $f(xy)-f(x)-f(y)= \alpha xy$& Corollary \ref{cor_exp_real} \\ \hline
$f, \alpha \colon S\to \mathbb{F}$, $(S, \cdot)\leq (\mathbb{F}^{\times}, \cdot)$ & $f(xy)-f(x)-f(y)= \alpha(xy)$& Proposition \ref{prop2} \\ \hline
$f, \alpha \colon \mathbb{F}\to \mathbb{F}$ & $f(x+y)-f(x)-f(y)= \alpha(xy)$& Proposition \ref{prop3} \\ \hline 
$f, \alpha \colon \mathbb{G}\to \mathbb{C}$ & $f(x\cdot y)-f(x)-f(y)= \alpha(x)(y)$& Theorem 
\ref{thm_exp_sol} 
\\ \hline \hline 
$f, \alpha \colon G\to \mathbb{C}$ & $f(xy)-f(x)f(y)= \alpha(xy)$& Proposition \ref{prop4} \\ \hline
$f, \alpha \colon \mathbb{K}^{\times}\to \mathbb{K}$ & $f(xy)-f(x)f(y)= \alpha(xy)$ & Corollary \ref{cor11} \\ \hline
$f, \alpha \colon \mathbb{K}\to \mathbb{K}$ & $f(x+y)-f(x)f(y)= \alpha(x+y)$ & Corollary \ref{cor12} \\ \hline
$f, \alpha \colon G\to \mathbb{C}$ & $f(x\cdot y)-f(x)f(y)= \alpha(x)\alpha(y)$& Theorem \ref{thm_exp_difference_prod} \\ \hline

\end{tabular}

\begin{ackn}
The research of E.~Gselmann  was supported by the János Bolyai Research Scholarship of the Hungarian Academy of Sciences. 
\end{ackn}


\noindent
\textbf{Eszter Gselmann}\\
Department of Analysis\\
University of Debrecen\\
H-4002 Debrecen, P.O.Box 400\\
Hungary\\
E-mail: \href{mailto:gselmann@science.unideb.hu}{gselmann@science.unideb.hu}\\
ORCID: \href{https://orcid.org/0000-0002-1708-2570}{0000-0002-1708-2570}

\begin{multicols}{2}
\vspace{1cm}
\noindent
\textbf{Tomasz Małolepszy}\\
Institute of Mathematics\\
University of Zielona Góra\\
65-516 Zielona Góra, ul. prof.~Z.~Szafrana 4a\\
Poland\\
E-mail: \href{mailto:t.malolepszy@im.uz.zgora.pl}{t.malolepszy@im.uz.zgora.pl}

\vspace{1cm}
\noindent
\textbf{Janusz Matkowski}\\
Institute of Mathematics\\
University of Zielona Góra\\
65-516 Zielona Góra, ul. prof.~Z.~Szafrana 4a\\
Poland\\
E-mail: \href{mailto:j.matkowski@im.uz.zgora.pl}{j.matkowski@im.uz.zgora.pl}
\end{multicols}

\end{document}